\documentclass[12pt,a4paper]{amsart}
\usepackage{amsfonts, amssymb, amsmath, amsthm}
\usepackage{graphicx,latexsym,eufrak,mathrsfs}
\usepackage[hmargin=1.2in, vmargin=1in]{geometry}

\newtheorem{theorem}{Theorem}[section]
\newtheorem{proposition}[theorem]{Proposition}
\newtheorem{lemma}[theorem]{Lemma}
\newtheorem{corollary}[theorem]{Corollary}

\theoremstyle{definition}

\theoremstyle{remark}

\numberwithin{equation}{section}

\begin{document}

\title[Eigenfunction expansions in $\mathbb R^n$] {Eigenfunction expansions of ultradifferentiable functions and ultradistributions in $\mathbb R^n$}

\author[\DJ.Vu\v{c}kovi\'{c}]{\DJ or\dj e Vu\v{c}kovi\'{c}}
\address{Department of Mathematics\\ Ghent University\\ Krijgslaan 281 Gebouw S22, 9000 Gent, Belgium}
\email{dordev@cage.UGent.be}

\author[J. Vindas]{Jasson Vindas}
\address{Department of Mathematics, Ghent University, Krijgslaan 281 Gebouw S22, 9000 Gent, Belgium}
\email{jvindas@cage.UGent.be}
\thanks{The authors gratefully acknowledge support by Ghent University, through the BOF-grant 01N01014.}


\dedicatory{Dedicated to the memory of Prof. Todor Gramchev}

\begin{abstract} We obtain a characterization of ${\mathcal S}^{\{M_p\}}_{\{M_p\}}(\mathbb R^n)$ and $\mathcal {S}^{(M_p)}_{(M_p)}(\mathbb {R}^n)$, the general Gelfand-Shilov spaces of ultradifferentiable functions of Roumieu and Beurling type, in terms of decay estimates for the Fourier coefficients of their elements with respect to eigenfunction expansions associated to normal globally elliptic differential operators of Shubin type. Moreover, we show that the eigenfunctions of such operators are absolute Schauder bases for these spaces of ultradifferentiable functions. Our characterization extends earlier results by Gramchev et al. (Proc. Amer. Math. Soc. 139 (2011), 4361--4368) for Gevrey weight sequences. It also generalizes to $\mathbb{R}^{n}$ recent results by Dasgupta and Ruzhansky which were obtained in the setting of compact manifolds.
 
\end{abstract} 

\maketitle

\section{Introduction}

Back in 1969 Seeley characterized \cite{See} real analytic functions on a compact analytic manifold via the decay of their Fourier coefficients with respect to eigenfunction expansions associated to a normal analytic elliptic differential operator. In recent times, this result by Seeley has attracted much attention and has been generalized in several directions. In a recent article \cite{dasgupta-ruzhansky2015}, Dasgupta and Ruzhansky extended Seeley's work and achieved the eigenfunction expansion characterization of Denjoy-Carleman classes of ultradifferentiable functions, of both Roumieu and Beurling type, and the corresponding ultradistribution spaces on a compact analytic manifold. See also \cite{dasgupta-ruzhansky2014} for Gevrey classes on compact Lie groups. 

Such results have also a global Euclidean counterpart. In this setting, it is natural to consider differential operators of Shubin type, that is, differential operators with polynomial coefficients
\begin{equation}
\label{1.3}P=\sum_{|\alpha|+|\beta|\leq m} c_{\alpha\beta} x^{\beta} D^{\alpha}, \quad D^{\alpha}=(-i\partial_x)^{\alpha}.
\end{equation}
In \cite{Pilip} Gramchev, Pilipovi\'{c}, and Rodino used this type of operators to give an analogue to Seeley's result for some classes of Gelfand-Shilov spaces.

The aim of this paper is to extend the results from \cite{Pilip} by supplying a characterization of the general Gelfand-Shilov spaces ${\mathcal S}^{\{M_p\}}(\mathbb R^n)={\mathcal S}^{\{M_p\}}_{\{M_p\}}(\mathbb R^n)$  and ${\mathcal S}^{(M_p)}(\mathbb R^n)=\mathcal {S}^{(M_p)}_{(M_p)}(\mathbb {R}^n)$ of ultradifferentiable functions of Roumieu and Beurling type \cite{PilipovicK,chung-chung-kim1996,Friedmanbook,GS,Langen}. Our characterization is as follows. We refer to Section \ref{preliminaries} for the notation. Note that if $P$ is globally elliptic and normal ($PP^{\ast}=P^{\ast}P$), then there is an orthonormal basis of $L^{2}(\mathbb{R}^{n})$ consisting of eigenfunctions of $P$. Properties of the Shubin type operators are very well explained in the textbooks \cite{Rodino,Shubin}. Our assumptions on the weight sequence are the standard $(M.1)$ and $(M.2)'$ Komatsu's conditions (logarithmic convexity and stability under differential operators \cite{Komatsu}), together with the essential assumption:
\begin{align}
\label{assumption}
\sqrt{p!}\leq C_{l} l^{p}M_{p}, \quad \forall p\in\mathbb{N}_{0} \quad &(\mbox{Roumieu case: for some }  l,C_{l}>0)
\\&
\nonumber
(\mbox{Beurling case: for all } l>0 \mbox{ there is } C_{l}>0).
\end{align}
The function $M$ below stands for the associated function of the weight sequence (cf. Section \ref{preliminaries}).
\begin{theorem}\label{mainintro}
Let $P$ be a normal globally elliptic differential operator of Shubin type $(\ref{1.3})$ and let  $\{u_j:j\in\mathbb N\}$ be an orthonormal basis of $L^{2}(\mathbb{R}^{n})$ consisting of eigenfunctions of $P$. Let $f\in L^{2}(\mathbb{R}^{n})$ have eigenfunction expansion
$$
f=\sum_{j=1}^{\infty} a_j u_j.
$$
Suppose that the weight sequence $M_p$ satisfies $(M.1)$, $(M.2)'$, and $(\ref{assumption})$.
Then,
\begin{enumerate}
\item [$(i)$] $f\in{\mathcal S}^{\{M_p\}}(\mathbb R^n)$ if and only if there are $\lambda>0$ and $C_\lambda>0$ such that
\begin{equation}
\label{vazno1}
|a_j|\leq C_\lambda e^{-M( \lambda j^{\frac{1}{2n}})}, \quad j\in\mathbb{N}.
\end{equation}
\item [$(ii)$] $f\in{\mathcal S}^{(M_p)}(\mathbb R^n)$ if and only if the estimate $(\ref{vazno1})$ holds for each $\lambda>0$.
\end{enumerate}
\end{theorem}

Consequently, the global $M_p$ regularity and decay of a function $f$ are completely determined by the decay of its coefficients $a_j$. 
Since for Gevrey sequences $M_p=(p!)^{\mu}$ the associated function $M(t)\asymp |t|^{1/\mu}$ \cite{GS}, our result includes as particular instances those from \cite{Pilip}. In the special case of the harmonic oscillator
$$
-\Delta+|x|^2,
$$
the eigenfunctions are given by the Hermite functions; Theorem \ref{mainintro} thus also recovers well-known results for Hermite expansions \cite{PilipovicK,Langen,zhang1963} (see also \cite{KO2015}).

It is important to point out that Theorem \ref{mainintro} does not reveal all topological information involved in the problem. In fact, in Section \ref{section eigenexpansions ultradifferentiable functions} we prove a much stronger result, namely, we shall show that the eigenfunctions $u_j$ are absolute Schauder bases for  ${\mathcal {S}}^{\ast}(\mathbb R^n)$, where $\ast=\{M_p\}$ or $(M_p)$, and that these spaces become (tamely) isomorphic as topological vector spaces to sequence spaces canonically defined by the estimates (\ref{vazno1}). 
This will easily yield an eigenfunction expansion characterization of the ultradistribution spaces ${\mathcal{S}^{\ast}}' (\mathbb R^n)$. In Section \ref{section iterates} we characterize $\mathcal{S}^{\ast}(\mathbb{R}^{n})$ via iterates of $P$; the characterization leads to the ensuing regularity result for solutions to the equation $Pu=f$.
\begin{theorem}
\label{regularity result}
Let $P$ be a globally elliptic operator of Shubin type $(\ref{1.3})$ and let $M_p$ satisfy $(M.1)$, $(M.2)'$, and $(\ref{assumption})$. If $u\in{\mathcal{S}^{\ast}}'(\mathbb R^n)$ is a solution to $Pu=f$ and $f\in{\mathcal{S}^{\ast}}(\mathbb R^n)$, then also $u\in{\mathcal{S}^{\ast}}(\mathbb R^n)$.
\end{theorem}

\section{Preliminaries}
\label{preliminaries}
Throughout the article we shall fix the differential operator $P$ of Shubin type having order $m$, explicitly given by (\ref{1.3}). Additional assumptions on $P$ will be imposed when needed. We also fix a positive weight sequence $M_p$ with $M_0=1$. Besides (\ref{assumption}), we make use of the following conditions on $M_p$:
\begin{itemize}
\item [$(M.1)\:$] $M^{2}_{p}\leq M_{p-1}M_{p+1},$  $p\geq 1$.
\item [$(M.2)'$]$M_{p+1}\leq A H^p M_p$, $p\in\mathbb{N}_0$, for some $A>0$ and $H\geq 1$.
\item [$(M.2)\:$] $ \displaystyle M_{p}\leq A H^p\min_{1\leq q\leq p} \{M_{q} M_{p-q}\},$ $p\in\mathbb{N}$, for some $A>0$ and $H\geq 1$.
\end{itemize}
Note that the Gevrey sequences $M_p=(p!)^{\mu}$ satisfy all of the above properties, if $\nu\geq 1/2$ in the Roumieu case of (\ref{assumption}) and $\mu>1/2$ in the Beurling case of (\ref{assumption}).  The associated function of $M_p$ is
 $$
M(t):=\sup_{p\in\mathbb{N}_0}\log\frac{t^p}{M_p},\quad t>0.
$$
We refer to \cite{Komatsu} for the meaning of $(M.1)$, $(M.2)'$, and $(M.2)$, and their translation into properties of the associated function. 

We now derive a simple but very useful relation for sequences fulfilling $(M.1)$ and (\ref{assumption}). This relation plays a crucial role in Section \ref{section iterates}. Observe also that (\ref{romije}) obviously implies (\ref{assumption}).

\begin{lemma} The conditions $(M.1)$ and $(\ref{assumption})$ imply that 
\begin{align}
\label{romije}&\sqrt{p+1}\: \frac{M_p}{M_{p+1}}\leq r, \ \ p\in\mathbb{N}_{0}, \ \ \mbox{for some }r>0 \quad &(\mbox{Roumieu case}),
\\
&\nonumber
\lim_{p\to\infty}\sqrt{p+1} \: \frac{M_p}{M_{p+1}}=0.
&(\mbox{Beurling case})
\end{align}
\end{lemma}

\begin{proof}
Stirling's formula yields  $\sqrt{p+1}\leq C (\sqrt{p!})^{1/p}$. Using $(M.1)$, we conclude that $(M_p/M_{p+1})\leq M_{p}^{-1/p} $.
Thus, (\ref{assumption}) yields $\sqrt{p+1} M_p/M_{p+1}\leq CC_{l}^{1/p} l $.
\end{proof}
We define $\mathcal{S}^{\{M_p\}}(\mathbb{R}^{n})$ and $\mathcal{S}^{(M_p)}(\mathbb{R}^{n})$ as follows.
First introduce the Banach space ${\mathcal S}^{\{M_p\},h}_{L^{2}}$, $h>0$, consisting of all $f\in C^{\infty}(\mathbb{R}^{n})$ such that
\begin{equation}
\label{ultranorms}
\|f\|_{h}:=\sup_{\alpha,\beta\in\mathbb{N}^{n}_0}\frac{\|x^{\beta}\partial^{\alpha}f\|_{L^{2}(\mathbb{R}^{n})}}{h^{|\alpha|+|\beta|}M_{|\alpha|+|\beta|}}<\infty\ ;
\end{equation}
define then 
\begin{equation}
\label{eqspaces1}
{\mathcal{S}}^{\{M_p\}}(\mathbb R^n)=\bigcup_{h>0} {\mathcal S}^{\{M_p\},h}_{L^{2}} \quad \mbox{ and} \quad {\mathcal S}^{(M_p)}(\mathbb R^n)=\bigcap_{h>0} {\mathcal S}^{\{M_p\},h}_{L^{2}},
\end{equation}
the union and intersection having topological meaning as inductive and projective limits of Banach spaces. Under the assumption $(M.2)'$, these spaces are (DFS) and (FS) spaces, respectively. We use the notation $\ast=\{M_p\},(M_p)$ to treat the Roumieu and Beurling case simultaneously. As customary, one writes $\mathcal{S}^{\mu}_{\mu}(\mathbb{R}^{n})=\mathcal{S}^{\{M_p\}}(\mathbb{R}^{n})$ and $\Sigma^{\mu}_{\mu}(\mathbb{R}^{n})=\mathcal{S}^{(M_p)}(\mathbb{R}^{n})$ for the special case $M_{p}=(p!)^{\mu}$. Condition (\ref{assumption}) yields $\mathcal{S}^{1/2}_{1/2}(\mathbb{R}^{n})\subseteq\mathcal{S}^{\ast}(\mathbb{R}^{n})$, which ensures the non-triviality of these spaces. Naturally $P:\mathcal{S}^{\ast}(\mathbb{R}^{n})\to \mathcal{S}^{\ast}(\mathbb{R}^{n})$ becomes continuous if one assumes $(M.2)'$ and hence one can define $P$ on the ultradistribution space ${\mathcal{S}^{\ast}}'(\mathbb{R}^{n})$ via duality.

For the reader's convenience, we recall the definition of \emph{tame continuity} of linear mappings for graded Fr\'{e}chet spaces and inductive limits of Banach spaces. This notion is very important in the structure theory of Fr\'{e}chet spaces (see e.g. \cite{Vogt1987}). A graded Fr\'{e}chet space is a Fr\'{e}chet space together with a choice of a non-decreasing sequence of seminorms defining its topology. A continuous linear mapping $T:(E,|\ |_{j})\to (F,|\ |'_{j})$ between two graded Fr\'{e}chet spaces is called (linearly) tame if there are constants $L>0$ and $j_0$ such that $|T v|'_{Lj}\leq C_{j}|v|_{j}$, for all $j\geq j_0$ and $v\in E$. Tame continuity for (LB) spaces is defined similarly. Once one implicitly fixes the increasing sequences of Banach spaces, a mapping $T:E=\varinjlim_{j}E_{j}\to F=\varinjlim_{j} F_{j} $ is tamely continuous if there are $L$ and $j_0$  such that $\|T v\|_{F_{Lj}}\leq C_{j}\|v\|_{E_{j}}$, for all $j\geq j_0$ and $v\in E_{j}$. The meaning of a tame isomorphism is clear. 

In the next sections we always consider the grading of $\mathcal{S}^{\ast}(\mathbb{R}^{n})$ given by (\ref{eqspaces1}), that is, the one provided by the Banach spaces ${\mathcal S}^{\{M_p\},h}_{L^{2}}$. It is worth noticing that if  $(M.2)'$ holds, using the norms $\| \ \|_{L^{2}(\mathbb{R}^{n})}$ instead of $\| \ \|_{L^{\infty}(\mathbb{R}^{n})}$ in $(\ref{ultranorms})$ leads to an equivalent definition of $\mathcal{S}^{\ast}(\mathbb{R}^{n})$. Furthermore, that modified system of norms is tamely equivalent to (\ref{ultranorms}), as one easily verifies. 
We also remark that our definition of the norms (\ref{ultranorms}) does not separate between the behavior of derivatives and growth. On the other hand, if the sequence satisfies $(M.2)$, such behavior can be split and our system of norms becomes tamely equivalent to $\sup_{\alpha,\beta\in\mathbb{N}^{n}_0}\|x^{\beta}\partial^{\alpha}f\|_{L^{2}(\mathbb{R}^{n})}/(h^{|\alpha|+|\beta|}M_{|\alpha|}M_{|\beta|})$. However, $(M.2)$ plays basically no role in our arguments, we shall  therefore not impose it and we choose to use the family of norms (\ref{ultranorms}).

\section{Iterates of the operator and regularity of solutions}
\label{section iterates}
In this section we exploit the iterative approach from \cite{capiello-gramchev-rodino, Pilip, See} in order to obtain a structural characterization of $\mathcal{S}^{\ast}(\mathbb{R}^{n})$ in terms of the growth of the $L^2$ norms of the iterates of the operator $P$. The regularity result Theorem \ref{regularity result} will readily follow from Theorem \ref{iterates theorem 1} below. We point out that these ideas go back to seminal works by Komatsu \cite{komatsu1960,komatsu1962} and Kotak\'{e} and Narasimhan \cite{KN}. 

We begin by introducing function spaces associated to the iterates of $P$. At this point, we do not need any ellipticity assumption on $P$. For $h>0$, define the Banach space $\mathcal{S}^{\{M_p\},h}_{P}$ of all functions $f$ such $P^{p}f\in L^{2}(\mathbb{R}^{n})$ for all $p\in\mathbb{N}_0$ and
\begin{equation}
\label{normsP}
\|f\|_{P,h}:= \sup_{p\in\mathbb{N}_{0}} \frac{\|P^{p}f\|_{L^{2}(\mathbb{R}^{n})}}{h^{mp}M_{mp}}<\infty;
\end{equation}
set further,
\begin{equation*}
{\mathcal{S}}^{\{M_p\}}_{P}(\mathbb R^n)=\varinjlim_{h\to\infty} {\mathcal S}^{\{M_p\},h}_{P} \quad \mbox{ and } \quad {\mathcal S}^{(M_p)}_{P}(\mathbb R^n)=\varprojlim_{h\to 0^{+}} {\mathcal S}^{\{M_p\},h}_{P}.
\end{equation*}
We regard $\mathcal{S}^{\ast}_{P}(\mathbb{R}^{n})$ as spaces graded by the norms (\ref{normsP}). 
\begin{proposition} \label{iterates proposition 1} Suppose $M_p$ satisfies $(\ref{romije})$. Then, $\mathcal{S}^{\ast}(\mathbb{R}^{n})\subseteq \mathcal{S}^{\ast}_{P}(\mathbb{R}^{n})$ and the inclusion mapping $\mathcal{S}^{\ast}(\mathbb{R}^{n})\to \mathcal{S}^{\ast}_{P}(\mathbb{R}^{n})$ is tamely continuous.
\end{proposition}

\begin{proof}

Fix $f\in\mathcal{S}^{^{\{M_p\},h}}_{L^{2}}$ with $\|f\|_{h}=1$.
By employing the Leibniz formula, it is easy to see that 

\begin{align} 
 P^p u=
\sum_{(\boldsymbol{\alpha},\boldsymbol{\beta},\boldsymbol{\tau})\in{\mathcal C}_{p}} q_{\boldsymbol{\alpha},\boldsymbol{\beta},\boldsymbol{\tau}}(P) Q_{\boldsymbol{\alpha},\boldsymbol{\beta},\boldsymbol{\tau}}L_{\boldsymbol{\alpha},\boldsymbol{\beta},\boldsymbol{\tau}}(f),
 \end{align}
where the summation extends over the set ${\mathcal C}_{p}$ of all $(3p-1)$-tuples of multi-indices $(\boldsymbol{\alpha},\boldsymbol{\beta},\boldsymbol{\tau})=(\alpha_1,\dots,\alpha_p,\beta_1,\dots,\beta_p,\tau_1,\dots,\tau_{p-1})$ such that
$|\alpha_j|+|\beta_j|\leq m$ for each $j$,  $\tau_{j-1}\leq \alpha_{j}$ for $j=2,\dots, p$, and 
$
\tau_1+\dots +\tau_j\leq \beta_1+\dots+\beta_j$ for $j=1,2,\dots, {p-1},
$
and where
$$
q_{\boldsymbol{\alpha},\boldsymbol{\beta},\boldsymbol{\tau}}(P):= c_{\alpha_1,\beta_1} \prod_{j=2}^p  c_{\alpha_j,\beta_j}{{\alpha_j}\choose{\tau_{j-1}}} \:,
$$

$$
 Q_{\boldsymbol{\alpha},\boldsymbol{\beta},\boldsymbol{\tau}}:=\prod_{j=1}^{p-1} \frac{(\beta_{1}+\dots +\beta_{j}-\tau_1-\dots-\tau_{j-1} )!}{(\beta_{1}+\dots +\beta_{j}-\tau_1-\dots-\tau_{j-1}-\tau_{j})!}  \quad (\tau_{0}:=0),
$$
and the differential operator $L_{\boldsymbol{\alpha},\boldsymbol{\beta},\boldsymbol{\tau}}$ is given by
$$
L_{\boldsymbol{\alpha},\boldsymbol{\beta},\boldsymbol{\tau}}:=x^{\beta_{1}+\dots+\beta_{p}-\tau_{1}-\dots-\tau_{p-1}}  D^{\alpha_1+\dots+\alpha_p-\tau_1-\tau_2-\cdots-\tau_{p-1}}.
$$

Set $C_P=\max_{|\alpha|+|\beta|\leq m}\{|c_{\alpha,\beta}|\}$. First note that $|q_{\boldsymbol{\alpha},\boldsymbol{\beta},\boldsymbol{\tau}}(P)|\leq 2^{-m}(C_{P}2^{m})^{p}$, because of the well known estimate for binomial coefficients. We need an estimate on the number of elements of the set $\mathcal{C}_{p}$. The rough bound $|\mathcal{C}_{p}|\leq m^{-n} (2^{m+2n} m^{n})^{p}$ suffices for our purposes. Indeed, for a fixed $j$, the number of multi-indices such that $|\alpha_j|+|\beta_j|\leq m$ is $\sum_{\nu=0}^{m}\binom{\nu+2n-1}{\nu}\leq 2^{m+2n}$ and number of $\tau_j$ is less than $m^n$. We conclude then
$$
\|P^{p}f\|_{L^{2}(\mathbb{R}^{n})}\leq   m^{-n}2^{-m} (4^{1+n/m}m^{n/m}C_{P}^{1/m}h)^{mp}M_{mp} \max_{(\boldsymbol{\alpha},\boldsymbol{\beta},\boldsymbol{\tau})\in\mathcal{C}_{p}} Q'_{\boldsymbol{\alpha},\boldsymbol{\beta},\boldsymbol{\tau}}(h),
$$
where 
$$
Q'_{\boldsymbol{\alpha},\boldsymbol{\beta},\boldsymbol{\tau}}(h)=\frac{h^{|\boldsymbol{\alpha}|+|\boldsymbol{\beta}|-2|\boldsymbol{\tau}|}M_{|\boldsymbol{\alpha}|+|\boldsymbol{\beta}|-2|\boldsymbol{\tau|}}}{h^{pm}M_{mp}}\prod_{j=1}^{p-1} \frac{(|\beta_{1}|+\dots +|\beta_{j}|-|\tau_1|-\dots-|\tau_{j-1}| )!}{(|\beta_{1}|+\dots +|\beta_{j}|-|\tau_1|-\dots--|\tau_{j}|)!}\:.$$

We  now estimate each of these terms.  In order to treat both the Roumieu and Beurling case simultaneously, we rewrite the assumption (\ref{romije}) as $M_k/M_{k+1}\leq r_{k}/\sqrt{k+1}$, where in the Roumieu case $r_{k}=r$ and in the Beurling case $r_k$ is a non-increasing positive sequence tending to 0. We obtain
\begin{align*}
h^{|\boldsymbol{\alpha}|+|\boldsymbol{\beta}|-2|\boldsymbol{\tau}|-pm}\frac{M_{|\boldsymbol{\alpha}|+|\boldsymbol{\beta}|-2|\boldsymbol{\tau|}}}{M_{mp}}
&
\leq \frac{M_{mp-2|\boldsymbol{\tau|}}}{h^{2|\boldsymbol{\tau|}}M_{mp}}\prod_{k=|\boldsymbol{\alpha}|+|\boldsymbol{\beta}|-2|\boldsymbol{\tau}|}^{mp-2|\boldsymbol{\tau}|-1}\frac{r_{k}}{h}
\\
&
\leq \left(\prod_{k=|\boldsymbol{\alpha}|+|\boldsymbol{\beta}|-2|\boldsymbol{\tau}|}^{mp-1}\frac{r_{k}}{h}\right)
\left(\prod_{\nu=mp-2|\boldsymbol{\tau}|}^{mp-1} \frac{1}{\sqrt{\nu+1}}\right)
\\
&
\leq \left(\prod_{k=1}^{mp-|\boldsymbol{\alpha}|-|\boldsymbol{\beta}|+2|\boldsymbol{\tau}|}\frac{r_{k}}{h}\right)
\left(\prod_{\nu=mp-2|\boldsymbol{\tau}|}^{mp-1} \frac{1}{\sqrt{\nu+1}}\right).
\end{align*}
In the Beurling case we have that the sequence $\prod_{k=1}^{j}(r_k/h)$ is bounded by some $C'_{h}$ because $r_{k}\to0$. In the Roumieu case this sequence is bounded by $C_h'=1$ if we ask $h\geq r$ (we impose this condition in the Roumieu case in the rest of the proof). Further on, clearly

$$
\frac{(|\beta_{1}|+\dots +|\beta_{j}|-|\tau_1|-\dots-|\tau_{j-1}| )!}{(|\beta_{1}|+\dots +|\beta_{j}|-|\tau_1|-\dots--|\tau_{j}|)!}\leq \frac{ (mj)!}{{(mj-|\tau_{j}|)!}}.
$$
Making use of $\tau_j\leq \alpha_{j+1}$ and $\sum_{k=1}^j\tau_i\leq \sum_{k=1}^j\beta_k$, 
$$
mp-2|\tau_{j+1}|-2|\tau_{j+2}|-\cdots-2|\tau_{p-1}|\geq jm, $$
and hence ($\tau_{p}:=0$)
$$\prod_{\nu=mp-2|\boldsymbol{\tau}|}^{mp-1}\frac{1}{\sqrt{\nu+1}}=\prod _{j=1}^{p-1}\ \prod^{mp-2|\tau_{j+1}|-\dots-2|\tau_{p-1}|-1}_{\nu=mp-2|\tau_{j}|-\dots-2|\tau_{p-1}|} \frac{1}{\sqrt{\nu+1}}\leq \prod _{j=1}^{p-1}\sqrt{\frac{(mj-2|\tau_{j}|)!}{(mj)!}}\: .
$$
Since for $j>2$
$$\frac{jm(jm-1)\cdots (jm-|\tau_{j}|+1)}{\sqrt{jm(jm-1)\dots(jm-2|\tau_{j}|+1)}}\leq \left(\frac{jm}{(j-2)m}\right)^m\leq3^{m}, $$
we obtain that $Q'_{\boldsymbol{\alpha},\boldsymbol{\beta},\boldsymbol{\tau}}(h)\leq 3^{mp} (2m^{3/2})^{m}C'_{h}/27 $. Summarizing, in the Beurling case we have shown that $\|\cdot \|_{P,Lh}\leq C_{h}\|\cdot \|_{h}$ for all $h>0$ where $L=4^{1+n/m}3m^{n/m}C_{P}^{1/m}$, while in the Roumieu case such inequality is valid for all $h\geq r$. This establishes the claimed inclusion and its tame continuity.
\end{proof}

Our next goal is to show that actually $\mathcal{S}^{\ast}(\mathbb{R}^{n})=\mathcal{S}^{\ast}_{P}(\mathbb{R}^{n})$ whenever $P$ is globally elliptic. Recall \cite{Rodino,Shubin} that global ellipticity means that the principal symbol
\begin{equation}\label{global}
\sum_{|\alpha|+|\beta|=m} c_{\alpha\beta} x^{\beta}\xi^{\alpha}\not=0 \quad \mbox{for all }  \  (x,\xi)\not=(0,0).
\end{equation}

Our starting point is the same as in \cite{Pilip}, i.e., the interpolating inequality \cite[Prop. 4.1]{Pilip}
\begin{equation}
\label{interpolatingeq}
|f|_{pm+j}\leq |f|_{pm}+C |f|_{(p+1)m}+ C^{pm+j} ((pm+j)!)^{1/2} \|f\|_{L^2(\mathbb R^d)}\:,
\end{equation}
where $0<j<m$ and $1\leq C$, for the Sobolev type seminorms
$$
|f|_s:=\sum_{|\alpha|+|\beta|=s}\|x^{\beta}\partial^{\alpha}f\|_{L^2(\mathbb R^n)}.
$$
We consider the family of norms
\begin{equation}
\label{ultranorms2}
\|f\|_{h}'=\sup_{p\in\mathbb{N}_{0}} \frac{|f|_{pm}}{h^{pm}M_{pm}}, \quad h>0.
\end{equation}
\begin{proposition}
\label{iterates proposition 2} Under the assumptions $(M.2)'$ and $(\ref{assumption})$, the family of norms $(\ref{ultranorms})$ and $(\ref{ultranorms2})$ are tamely equivalent $($both as $h\to\infty$ and $h\to0^{+}$$)$.
\end{proposition}
\begin{proof}
Clearly, $\| \cdot \|'_{h}\leq 2^{2n-1}\| \cdot \|_{h/2}$ without any assumption on $M_p$. In the Roumieu case, a routine computation with the aid of (\ref{interpolatingeq}) shows that $\| \cdot \|_{H^{m}h}\leq C'_{h}\| \cdot \|_{h}'$ for all $h\geq Cl$ with  $C'_h=C(hAH^{(m-1)/2})^{m}+C_{l}+\max\{1,(r/h)^{m}\}$, where these are the constants occurring in $(M.2)'$, (\ref{assumption}), (\ref{romije}), and (\ref{interpolatingeq}). In the Beurling case we obtain $\| \cdot \|_{H^{m}h}\leq C'_{h}\| \cdot \|_{h}'$ for all $h\leq1$ with $C'_h=C(AH^{(m-1)/2})^{m}+C_{h/C}+\max\{1,(r/h)^{m}\}$ where again $r$ is an upper bound for $\sqrt{p+1}M_p/M_{p+1}$.
\end{proof}

We need the ensuing adapted version of \cite[Prop. 4.2]{Pilip}. Set
$$\sigma_p(f,h)=\frac{|f|_{mp}}{h^{mp} M_{mp}}, \quad p\in\mathbb{N}_{0},$$
so that $\sigma_{0}(f,h)=\|f\|_{L^{2}(\mathbb{R}^{2})}$. We also set $\sigma_{-1}(f,h)=0$.

\begin{lemma}\label{iterates lemma 1} Let $P$ be globally elliptic and suppose that $(\ref{romije})$ holds. There is a constant $C'$ depending only on the operator and having the following properties:
\begin{itemize}
\item [(i)] In the Roumieu case there is $h_{0}>0$ (depending only on $P$ and the weight sequence) such that for all $h\geq h_{0}$
\begin{equation}
\label{3.6} 
\sigma_{p+1}(f,h)\leq\frac{C'M_{pm}}{h^{m}M_{(p+1)m}}\sigma_p(Pf,h)+\frac{1}{3}(\sigma_p(f,h)+\sigma_{p-1}(f,h)+\sigma_0(f,h)).
\end{equation}
\item [(ii)] In the Beurling case there is a positive non-increasing sequence $r_p$ tending to 0, which depends only on  $P$ and the weight sequence, such that
\begin{equation}
 \label{3.6'} 
\sigma_{p+1}(f,h)\leq\frac{C'M_{pm}}{h^{m}M_{(p+1)m}}\sigma_p(Pf,h)+\frac{r_{p}}{3h^{m}}\sigma_p(f,h)+\frac{r_{p}}{3h^{2m}}\sigma_{p-1}(f,h)+\sigma_0(f,h) \frac{r_1\cdots r_{p}}{3h^{m(p+1)}}.
\end{equation}
\end{itemize}
\end{lemma}
\begin{proof} We closely follow the proof of \cite[Prop. 4.2]{Pilip} with the required modifications. First notice that   $P:Q^m(\mathbb R^n)\rightarrow L^2(\mathbb R^n)$ is Fredholm, where $Q^{m}(\mathbb{R}^{n})$ denotes the Sobolev type space consisting of functions with $\|u\|_{Q^{m}(\mathbb{R}^{n})}=\sum_{j=0}^{m}|u|_{j}<\infty$, and actually $\operatorname*{Ker} P$  is a finite dimensional subspace of the Schwartz space $\mathcal S(\mathbb R^n)$ \cite{Rodino}. We may therefore assume for the sake of simplicity that $\operatorname*{Ker} P=\{0\}$. Now, there is then a constant $C_{1}>0$ such that
\begin{equation}
\label{operatoreq}
\sum_{|\alpha|+|\beta|\leq m}  \|x^{\beta} D^{\alpha} f\|_{L^2(\mathbb R^n)}=\sum_{s=0}^m|f|_s\leq C_1 \|Pf\|_{L^2(\mathbb R^d)}.
\end{equation}
Estimating exactly as in the proof of \cite[Prop. 4.2]{Pilip} with the aid of commutators and (\ref{operatoreq}), 
\begin{equation*}
|f|_{(1+p)m}\leq C'|Pf|_{pm}+C_2( (pm)^{m/2}|f|_{pm}+ (pm)^m |f|_{(p-1)m}+ C_3^{p} ((p+1)m)!^{1/2} |f|_0),
\end{equation*}
where the constants depend only on the operator and we may assume they are $\geq1$. As in the proof of Proposition \ref{iterates proposition 1}, the condition (\ref{romije}) ensures the existence of a non-increasing sequence of positive numbers $r'_p$ such that 
$
\sqrt{p+1} M_p/M_{p+1}\leq r'_{p}$, $\forall p\in \mathbb{N}_{0},$
where in the Roumieu case we may take it to be constant $r'_p=r$ ($\geq 1$), while in the Beurling case $r'_{p}\to0^{+}$.
Hence, (\ref{3.6'}) holds with any non-increasing sequence $r_p$ majorizing the three sequences $ (3C_2)^{1/p}C_3b_p$, $3C_2b_{p}b_{p-1}$, and $3C_2b_{p}$, where $b_{p}=\prod_{\nu=pm}^{pm+m-1}r'_{\nu}$. In the Beurling case we can clearly choose $r_p\to 0^{+}$. In the Roumieu case (\ref{3.6}) holds if we select $h_{0}=(3C_2 C_3)^{1/m}r^{2}$.
\end{proof} 

We can now state and prove the main theorem of this section:

\begin{theorem}\label{iterates theorem 1} Let $P$ be globally elliptic and let $M_p$ satisfy $(M.1)$, $(M.2)'$, and $(\ref{assumption})$. We have that $\mathcal{S}^{\ast}_{P}(\mathbb{R}^{n})=\mathcal{S}^{\ast}(\mathbb{R}^{n})$ and they are tamely isomorphic.
\end{theorem}

\begin{proof} We start with the Beurling case. Since the sequence $r_{p}\searrow0$, we can find $p_h$ large enough such that (\ref{3.6}) holds for all $p\geq p_h$. We may assume that $r_1\geq1$. We keep $h\leq r_1$. For $p\leq p_h$, one gets from (\ref{3.6'})
$$
\sigma_{p}(f,h)\leq \frac{C'M_{(p-1)m}}{h^{m}M_{pm}}\sigma_{p-1}(Pf,h)+\frac{r_{1}}{3h^{m}}\sigma_{p-1}(f,h)+\frac{r_{1}}{3h^{2m}}\sigma_{p-2}(f,h)+\sigma_0(f,h) \frac{r_1^{p-1}}{3h^{mp}}.
$$
Iterating these two relations, one obtains 
\begin{align}
\label{in'}
\sigma_{p+1}(f,h)\leq &\frac{C_1}{h^{m}}\left(\sum_{q=p_h}^{p}\frac{M_{qm}}{M_{(q+1)m}}\sigma_q(Pf,h)+\sum_{q=0}^{p_h-1}\frac{C_1^{p_h-1-q}}{h^{(p_h-1-q)m}}\frac{M_{qm}}{M_{(q+1)m}}\sigma_q(Pf,h)\right)
\\
&\nonumber
\quad \quad
+\frac{C_1^{p_h}}{h^{p_hm}}\sigma_0(f,h),
\end{align}
where $C_1=\max{\{r_1,C'}\}$. Iterating once more, we have
\begin{equation}
\label{iterateeq1}
\sigma_{p+1}(f,h)\leq \frac{C_1^{p_h}}{h^{p_hm}}\sum_{s=0}^{p}{p\choose s}C_1^{s}\frac{\sigma_{0}(P^{s}f,h)}{h^{sm}M_{sm}},
\end{equation}
for $h\leq C_1$. In fact, we check the latter inequality inductively. The assumption $(M.1)$ yields 
$M_{qm}/M_{(q+1)m}\leq M_{sm}/M_{(s+1)m}$ if $s\leq q$.
By (\ref{in'}), (\ref{iterateeq1}) for $q\leq p$, and $h\leq C_1$
\begin{align*}
\sigma_{p+1}(f,h)&\leq  \frac{C_1^{p_h}}{h^{p_hm}}\left(\sum_{q=0}^{p}\frac{C_1M_{qm}}{h^{m}M_{(q+1)m}} \sigma_q(Pf,h)+\sigma_0(f,h))\right)
\\
&
=
\frac{C_1^{p_h}}{h^{p_hm}}\left(\sigma_0(f,h))+\sum_{q=0}^{p}\frac{M_{qm}}{M_{(q+1)m}}\sum_{s=0}^{q} {q\choose s}\frac{C_1^{s+1}}{h^{m(s+1)}}\frac{\sigma_0(P^{s+1}f,h)}{M_{sm}}\right)
\\
&
\leq
\frac{C_1^{p_h}}{h^{p_hm}}\left(\sigma_0(f,h))+\sum_{s=0}^{p} {p+1\choose s+1}\frac{C_1^{s+1}}{h^{m(s+1)}}\frac{\sigma_0(P^{s+1}f,h)}{M_{(s+1)m}}\right),
\end{align*}
which shows  (\ref{iterateeq1}). It now follows immediately from (\ref{iterateeq1}) that 
$ \|\cdot\|'_{hL} \leq C'_h \|\cdot \|_{P,h}$ for all $h\leq r_1$, where $C'_{h}=(h^{-m}C_1)^{p_h}$ and $L= (1+C_1)^{1/m}$. Combining this with Proposition \ref{iterates proposition 2}, we obtain that $\mathcal{S}^{(M_p)}_{P}(\mathbb{R}^{n})\subseteq \mathcal{S}^{(M_p)}(\mathbb{R}^{n})$ and the inclusion mapping $\mathcal{S}^{(M_p)}_{P}(\mathbb{R}^{n})\to \mathcal{S}^{(M_p)}(\mathbb{R}^{n})$ is tamely continuous. The rest was already shown in Proposition \ref{iterates proposition 1}, which completes the proof in the Beurling case.

The Roumieu case is simpler. We keep $h\geq h_0$, where $h_0$ is the constant occurring in part $(i)$ of Lemma \ref{iterates lemma 1}. Iterating (\ref{3.6}) in an analogous way as in the Beurling case, we obtain
$$
\sigma_{p+1}(f,h)\leq \sum_{s=0}^{p}{p\choose s}(C')^{s}\frac{\|P^{s}f\|_{L^{2}(\mathbb{R}^{n})}}{h^{sm}M_{sm}},
$$
which implies that $ \|\cdot\|'_{hL} \leq \|\cdot \|_{P,h}$ for all $h\geq h_0$, where $L= (1+C')^{1/m}$. The rest follows once again from Proposition \ref{iterates proposition 1} and Proposition \ref{iterates proposition 2}.
 \end{proof} 
 
 Theorem \ref{regularity result} is now an easy consequence of Theorem \ref{iterates theorem 1}. In fact, if $Pu=f\in\mathcal{S}^{\ast}(\mathbb{R}^{n})$, the standard result \cite{Rodino} yields membership to the Schwartz space, that is, $u\in\mathcal{S}(\mathbb{R}^{n})$. Since $\|u\|_{P,h}=\max\{\|u\|_{L^{2}(\mathbb{R}^{n})}, \|f\|_{P,h} \}$, we conclude $u\in\mathcal{S}^{\ast}_{P}(\mathbb{R}^{n})=\mathcal{S}^{\ast}(\mathbb{R}^{n})$. As a corollary, we recover a result first observed in \cite{capiello-gramchev-rodino}: \emph{If $P$ is globally elliptic then all its eigenfunctions belong to $\mathcal{S}^{\{(p!)^{1/2}\}}(\mathbb{R}^{n})=\mathcal{S}^{1/2}_{1/2}(\mathbb{R}^{n})$.} Actually, we can strengthen this result by adding a bound on the partial derivatives of the eigenfunctions, the ensuing result is a direct corollary of the tame isomorphism established in this section (and inspection in the constants occurring in the proofs of the results for the Roumieu case) .
\begin{corollary}
\label{iterates corollary 1} Let $P$ be globally elliptic. There are constants $L_1$ and $L_2$ depending merely on $P$ such that if $u$ is a solution to $Pu=\lambda u$, $\lambda\in\mathbb{C}$, then
\begin{itemize}
\item [$(i)$]$\|x^{\beta}\partial^{\alpha} u\|_{L^{2}(\mathbb{R}^{n})}\leq L_{1}^{|\alpha|+|\beta|} (\alpha!\beta!)^{1/2} \| u\|_{L^{2}(\mathbb{R}^{n})}$ if $\lambda=0$.
\item [$(ii)$] $\|x^{\beta}\partial^{\alpha} u\|_{L^{2}(\mathbb{R}^{n})}\leq L_2|\lambda|(L_1|\lambda|^{\frac{1}{m}})^{|\alpha|+|\beta|} (\alpha!\beta!)^{1/2} \| u\|_{L^{2}(\mathbb{R}^{n})}$ if $\lambda\neq 0$.
\end{itemize}
\end{corollary}
\section{Eigenfunction expansions}
\label{section eigenexpansions ultradifferentiable functions}
We now study eigenfunction expansions of ultradifferentiable functions and ultradistributions. 

Through the rest of the article we assume that $P$ is globally elliptic and normal. As pointed out in the Introduction, these two conditions on $P$ guarantee the existence of an orthonormal bases of $L^{2}({\mathbb{R}^{n}})$ consisting of eigenfunctions of $P$. We fix such an orthonormal basis of eigenfunctions $\{u_j: j\in\mathbb N\}$. For each $j$, let $\lambda_{j}$ be the eigenvalue corresponding to $u_j$. Since $PP^{\ast}$ is positive and self-adjoint, and has order $2m$ and eigenvalues $|\lambda_{j}|^{2}$, the Weyl asymptotic formula yields
\begin{equation}\label{tilda} 
|\lambda_j|\sim B j^{\frac{m}{2n}},
\end{equation} 
where the constant $B$ depends on the symbol of $PP^{\ast}$, 
see \cite{boggiatto, Rodino, Shubin} for details. We introduce two (graded) sequence spaces suggested by the inequalities (\ref{vazno1}), that is, the (LB) space
$$
\Lambda^{\{M_p\}}_{n}=\{(a_j)_{j\in\mathbb{N}}\in \mathbb{C}^{\mathbb{N}}:\ \sup_{j\in\mathbb{N}}|a_j|e^{M( j^{\frac{1}{2n}}/h)}<\infty\ \mbox{ for some }h>0  \},
$$
and the Fr\'{e}chet space
$$
\Lambda^{(M_p)}_{n}=\{(a_j)_{j\in\mathbb{N}}\in \mathbb{C}^{\mathbb{N}}:\ \sup_{j\in\mathbb{N}}|a_j|e^{M( j^{\frac{1}{2n}}/h)}<\infty\ \mbox{ for every }h>0  \}.
$$
The concept of absolute Schauder bases for locally convex spaces is defined in \cite[p. 340]{meise-vogt}.
 
\begin{theorem}
\label{eigenexpansions theorem 1}
Let $P$ be normal and globally elliptic and let $M_p$ satisfy $(M.1)$, $(M.2)'$, and $(\ref{assumption})$. The mapping 
$$f\mapsto ((f,u_j)_{L^{2}(\mathbb{R}^{n})})_{j\in\mathbb{N}}$$ is a tame isomorphism from $\mathcal{S}^{\ast}(\mathbb{R}^{n})$ onto  $\Lambda^{\ast}_{n}$. Moreover, the set of eigenfunctions $\{u_{j}: j\in\mathbb{N}\}$ is an absolute Schauder basis for $\mathcal{S}^{\ast}(\mathbb{R}^{n})$.
\end{theorem}

\begin{proof} That $\{u_{j}:\: j\in\mathbb{N}_{0}\}$ is an absolute Schauder basis of $\mathcal{S}^{\ast}(\mathbb{R}^{n})$ follows readily from the first assertion and the fact that it is an orthonormal basis of $L^{2}(\mathbb{R}^{n})$, we leave details to the reader. Because of Theorem \ref{iterates theorem 1}, we can work with the system of norms (\ref{normsP}). Define the function
\begin{equation*}
\widetilde{M}(t):=\sup_{p\in\mathbb{N}_0}\log\frac{t^{mp}}{M_{mp}},\quad t> 0,
\end{equation*}
and notice that $\widetilde{M}(t)\leq M(t)$ and $M(t)\leq \widetilde{M}(H^{m}t)+\log (A^{m}H^{\frac{(m+2)(m-1)}{2}}) $, as one readily verifies with the aid of $(M.2)'$. Thus, using $M$ for the definition of $\Lambda^{\ast}_{n}$ is tamely equivalent to using the function $\widetilde{M}$. Furthermore, the system of norms $\|(a_j)_{j}\|_{\infty,h}:=\sup_{j\in\mathbb{N}}|a_j|e^{\widetilde{M}( j^{\frac{1}{2n}}/h)}$ for $\Lambda^{\ast}_{n}$ is tamely equivalent to $\|(a_j)_{j}\|_{2,h}:=\|(a_j e^{\widetilde{M}( j^{\frac{1}{2n}}/h)})_{j}\|_{\ell^{2}(\mathbb{N})}$. In fact, we trivially have $\|(a_j)_{j}\|_{\infty,h}\leq \|(a_j)_{j}\|_{2,h}$ for all $h>0$. On the other hand, the sequence $M_{mp}$ satisfies $M_{(p+1)m}\leq (AH^{\frac{m+1}{2}})^{m}H^{pm^{2}} M_{pm}$, and applying \cite[Prop. 3.4, p. 50] {Komatsu} to $M_{pm}$, we obtain 
$$
e^{\widetilde{M}(t)}\leq A^{2n}H^{n(m+1)}\: \frac{e^{\widetilde{M}(H^{2n}t)}}{t^{2n}}, \quad t>0.
$$
The latter inequality implies that $\|(a_j)_{j}\|_{2,h}\leq \|(a_j)_{j}\|_{\infty,H^{-2n}h}(AhH^{\frac{m+1}{2}})^{2n} \pi/\sqrt{6}$ for all $h>0$, showing the claimed tame equivalence. Write now $a_{j}=(f,u_j)_{L^{2}(\mathbb{R}^{n})}$ and let $d=\operatorname*{dim}(\operatorname*{Ker} P)$. Employing the Weyl asymptotics (\ref{tilda}), we have
\begin{equation*}
B_1^{2}\|P^p f\|^2_{L^2(\mathbb R^n)}\leq \sum_{j=1}^{\infty} j^{\frac{mp}{n}} |a_j|^2\leq d^{\frac{mp}{n}}\|f\|^{2}_{L^{2}(\mathbb{R}^{2})}+B_2^{2} \|P^p f\|^2_{L^2(\mathbb R^n)},
\end{equation*}
whence $B_{1}\|f\|_{P,h}\leq \|(a_{j})_{j}\|_{2,h}$ and $\|(a_{j})_{j}\|_{\infty,h}\leq \|f\|_{P,h}\sqrt{B^{2}_{2}+e^{2\widetilde{M}(d^{\frac{1}{2n}}/h)}}$ for all $h>0$. This concludes the proof of the theorem.
\end{proof}
Observe that if $(M.2)'$ holds, the strong duals of $\Lambda^{\ast}_{n}$ are precisely
$$
(\Lambda^{\{M_p\}}_{n})'=\{(a_j)_{j\in\mathbb{N}}\in \mathbb{C}^{\mathbb{N}}:\ \sup_{j\in\mathbb{N}}|a_j|e^{-M( j^{\frac{1}{2n}}/h)}<\infty\ \mbox{ for all }h>0  \},
$$
and
$$
(\Lambda^{(M_p)}_{n})'=\{(a_j)_{j\in\mathbb{N}}\in \mathbb{C}^{\mathbb{N}}:\ \sup_{j\in\mathbb{N}}|a_j|e^{-M( j^{\frac{1}{2n}}/h)}<\infty\ \mbox{ for some }h>0  \}.
$$
Therefore, we obtain the following corollary from Theorem \ref{eigenexpansions theorem 1} for ultradistributions. Note that the ultradistributional evaluation $\langle f,\overline{u}_{j} \rangle={}_{\mathcal{S^{\ast}}'}\langle f,\overline{u}_{j} \rangle_{\mathcal{S^{\ast}}}$ is well-defined in view of Corollary \ref{iterates corollary 1}.
\begin{corollary}
\label{eigenexpansions corollary 1}
Under the assumptions of Theorem \ref{eigenexpansions theorem 1}, every ultradistribution $f\in{\mathcal{S}^{\ast}}'(\mathbb{R}^{n})$ has eigenfunction expansion
$$
f=\sum_{j=1}^{\infty}a_j u_j, \quad \quad a_j=\langle f,\overline{u}_{j} \rangle.
$$
Furthermore, $\{u_j: j\in\mathbb{N}\}$ is an absolute Schauder basis for ${\mathcal{S}^{\ast}}'(\mathbb{R}^{n})$ and the mapping $f\mapsto (a_j)_{j\in\mathbb{N}}$ is a tame isomorphism from ${\mathcal{S}^{\ast}}'(\mathbb{R}^{n})$ onto ${\Lambda^{\ast}_{n}}'$.
\end{corollary}

We end this article with a specialized version of Corollary \ref{iterates corollary 1}. We mention that one may also deduce Theorem \ref{eigenexpansions theorem 1} from these bounds on the derivatives of the eigenfunctions, but we omit details for the sake of brevity. 

\begin{corollary}
\label{eigenexpansions corollary 1} Let $P$ be normal and globally elliptic. Then, there is a constant $\ell=\ell_{P}$ such that 
$$\|x^{\beta}\partial^{\alpha} u\|_{L^{2}(\mathbb{R}^{n})}\leq j^{\frac{m+|\alpha|+|\beta|}{2n}}\ell^{|\alpha|+|\beta|} (\alpha!\beta!)^{1/2} \| u\|_{L^{2}(\mathbb{R}^{n})},
$$
for each eigenfunction u with $Pu=\lambda_{j}u$.
\end{corollary}
\begin{proof} Apply Corollary \ref{iterates corollary 1} and the asymptotic estimate (\ref{tilda}). 
\end{proof}


\smallskip

\begin{thebibliography}{99}

\bibitem{boggiatto} P.~Boggiatto, E.~Buzano, L.~Rodino, \emph{Global hypoellipticity and spectral theory,} Math. Res. 92, Akademie Verlag, Berlin, 1996.

\bibitem{capiello-gramchev-rodino} M. Cappiello, T. Gramchev, L. Rodino, \emph{Super-exponential decay and holomorphic extensions for semilinear equations with polynomial coefficients,} J. Funct. Anal. \textbf{237} (2006), 634--654. 

\bibitem{PilipovicK} R.~Carmichael, A.~Kami\'nski, S.~Pilipovi\'c, \textit{Boundary values and convolution in ultradistribution spaces}, World Scientific Publishing Co. Pte. Ltd., Hackensack, NJ, 2007.

\bibitem{chung-chung-kim1996} J.~Chung, S.-Y.~Chung, D.~Kim, \emph{Characterizations of the Gel'fand-Shilov spaces via Fourier transforms,} Proc. Amer. Math. Soc. \textbf{124} (1996), 2101--2108.


\bibitem{dasgupta-ruzhansky2014} A.~Dasgupta, M. Ruzhansky, \emph{Gevrey functions and ultradistributions on compact Lie groups and homogeneous spaces,} Bull. Sci. Math. \textbf{138} (2014), 756--782.

\bibitem{dasgupta-ruzhansky2015} A.~Dasgupta, M.~Ruzhansky, \emph{Eigenfunction expansions of ultradifferentiable functions and ultradistributions,} Trans. Amer. Math. Soc. (2016), doi:10.1090/tran/6765.


\bibitem{Friedmanbook} A.~Friedman, \emph{
Generalized functions and partial differential equations,} Prentice-Hall, Inc., Englewood Cliffs, N.J., 1963.

\bibitem{GS} I.~M. ~Gel'fand, G.~E.~Shilov, \emph{Generalized functions. Vol. 2. Spaces of fundamental and generalized functions}, Academic press, New York-London, 1968.
 
\bibitem{Pilip} T.~Gramchev, S.~Pilipovi\'{c}, L.~Rodino, \emph{Eigenfunction expansions in $\mathbb R^n$},  Proc. Amer. Math. Soc. \textbf{139} (2011), 4361--4368.

\bibitem{KO2015} T. Kagawa, Y. Oka, \emph{A characterization of the generalized functions via the special Hermite expansions},  J. Pseudo-Differ. Oper. Appl., in press, doi:10.1007/s11868-015-0135-7.
  
\bibitem{komatsu1960} H. Komatsu, \emph{A characterization of real analytic functions,} Proc. Japan Acad. \textbf{36} (1960), 90--93.  

\bibitem{komatsu1962}  H. Komatsu, \emph{A proof of Kotak\'{e} and Narasimhan's theorem,} Proc. Japan Acad. \textbf{38} (1962), 615--618.

\bibitem{Komatsu} H.~Komatsu, \emph{Ultradistributions I. Structure theorems and a characterization}, J. Fac. Sci. Tokyo Sect. IA Math. \textbf{20} (1973), 25--105.

\bibitem{KN}T.~Kotak\'{e}, M.~S.~Narasimhan, \emph{Regularity theorems for fractional powers of a linear elliptic operator,} Bull. Soc. Math. France \textbf{90} (1962), 449--471.

\bibitem{Langen} M.~R.~Langenbruch, \emph {Hermite functions and weighted spaces of generalized functions}, Manuscripta Math. \textbf{119} (2006), 269--285.


\bibitem{meise-vogt} R. Meise, D. Vogt, \emph{Introduction to functional analysis,} The Clarendon Press, Oxford University Press, New York, 1997.

\bibitem{Rodino} F.~Nicola, L.~Rodino, \emph {Global pseudo-differential calculus on Euclidean spaces}, Birkh\"{a}user Verlag, Basel, 2010. 


\bibitem{See} R.~T.~Seeley, \emph{Eigenfunction expansions of analytic functions}, Proc. Amer.  Math. Soc. \textbf{21} (1969), 734--738.

\bibitem{Shubin} M.~A.~Shubin, \emph{Pseudodifferential operators and spectral theory,} 
Springer-Verlag, Berlin, 1987.

\bibitem{Vogt1987} D.~Vogt, \emph{Tame spaces and power series spaces}, Math. Z. \textbf{196} (1987), 523--536.

\bibitem{zhang1963} G.-Z.~Zhang, \emph{Theory of distributions of $\mathcal{S}$ type and pansions,} Acta Math. Sinica \textbf{13} (1963), 193--203 (Chinese); translation in Chinese Math. Acta. \textbf{4} (1963), 211--221.
\end{thebibliography}
\end{document}